\documentclass[aop,MSNbibl,dvips]{arximspdf}
\usepackage{mathrsfs}

\doi{10.1214/10-AOP613}
\volume{40}
\issue{1}
\pubyear{2012}
\firstpage{1}
\lastpage{18}

\makeatletter
\newcommand{\fraca}[2]{{#1/#2}}
\newcommand{\fracb}[2]{{(#1)/#2}}
\newcommand{\fracc}[2]{{#1/(#2)}}
\newcommand{\fracd}[2]{{(#1/#2)}}
\newcommand{\frace}[2]{{(#1)/(#2)}}
\newcommand{\eqref}[1]{(\ref{#1})}
\newcommand{\iint}{\int\!\!\int}
\renewcommand{\epsilon}{\varepsilon}
\newtheorem{theorem}{Theorem}
\newtheorem{lemma}{Lemma}
\newtheorem{proposition}{Proposition}
\newproclaim{remark}{Remark}
\newproclaim{example}{Example}
\makeatother

\begin{document}
\begin{frontmatter}

\title{Dimension result and KPZ formula for two-dimensional multiplicative cascade processes}
\runtitle{Dimension results for multiplicative cascade process}

\begin{aug}
\author{\fnms{Xiong} \snm{Jin}\corref{}\ead[label=e1]{xiongjin82@gmail.com}}
\runauthor{X. Jin}
\affiliation{CNRS \& Universit\'{e} Paris XIII}
\address{LAGA\\
Institut Galil\'{e}e\\
Universit\'{e} Paris 13\\
93430 Villetaneuse\\
France\\
\printead{e1}} 
\end{aug}

\received{\smonth{9} \syear{2010}}
\revised{\smonth{9} \syear{2010}}

%
\begin{abstract}
We prove a Hausdorff dimension result for the image of two-dimensional
multiplicative cascade processes, and we obtain from this result a
KPZ-type formula which normally has one point of phase transition.
\end{abstract}

%

\end{frontmatter}

\section{Introduction}\label{intro}

The famous Knizhnik--Polyakov--Zamolodchikov formu\-la in quantum gravity
relates the conformal dimension $\Delta^0$ of any field operator of a
two-dimensional conformal field theory to the analogous dimension~$\Delta$
of the same operator when the theory is coupled to a
two-dimensional quantum gravity,
%
\begin{equation}\label{KPZ}
\Delta^0=\Delta+\frac{\gamma^2}{4}\Delta(\Delta-1)  \qquad \mbox{for } \gamma
=\sqrt{\frac{25-c}{6}}-\sqrt{\frac{1-c}{6}},
\end{equation}
where $c$ is the central charge of the conformal field theory. This
formula was first derived by Knizhnik, Polyakov and Zamolodchikov \cite
{KnPoZa88} in 1988 via Liouville quantum gravity in a light cone gauge,
building on a earlier work of Polyakov~\cite{Po87} in 1987. Shortly
after, David~\cite{Da88} provided an alternative heuristic derivation
of the KPZ formula by using Liouville field theory in the so-called
conformal gauge. The KPZ formula has a great influence on string theory
and conformal field theory, and it plays a core role in studying the
connections of two-dimensional quantum gravity to random planar maps,
two-dimensional lattice models, random matrix theory and
Schramm--Loewner evolution.

In a recent inspiring paper~\cite{DuSh08} Duplantier and Sheffield
provide (in a~mathematically rigorous way) a geometrical KPZ formula
under a similar framework as used in~\cite{Da88}. They relate the
Euclidean scaling exponent $x$ of a fractal subset of the domain
$D$\vadjust{\goodbreak}
(with respect to the Lebesgue measure) to the quantum scaling exponent
$\Delta$ of the same set (with respect to the Liouville quantum gravity
measure, that is, roughly speaking, the measure $e^{\gamma h}\,dz$ with~$h$
being the Gaussian free field on $D$). By using large deviation
estimates they prove that $x$ and $\Delta$ satisfy the same formula as
\eqref{KPZ} (replacing $\Delta^0$ by~$x$) for $\gamma\in[0,2)$.

Inspired by Duplantier and Sheffield's work, Benjamini and Schramm~\cite
{BeSc09} prove a Hausdorff dimension version of the geometrical KPZ
formula for random metrics built from Mandelbrot measures constructed
in~\cite{KaPe76}. Adapting Benjamini and Schramm's proof, Rhodes and
Vargas~\cite{RhVa08} prove a~similar result for one-dimensional
log-infinite divisible multifractal measures constructed in \cite
{BacMu03} and two-dimensional Gaussian multiplicative chaos constructed
in~\cite{Ka85,RoVa10} (like the Liouville quantum gravity measure
constructed in~\cite{DuSh08}). It is also worth mentioning that
following~\cite{DuSh08} there is the paper of David and Bauer \cite
{DaBa09} that gives a physicist's derivation of the geometrical KPZ
formula via heat kernel methods.

A common feature of the random measures mentioned above (Liouville
quantum gravity measure, Mandelbrot measure, log-infinite divisible
multifractal measure, etc.) is that they are all obtained through a
limiting procedure, and along the procedure the random densities that
are used to construct these measures can always be locally written as a
product of independent weights. These random measures nowadays are
mentioned as multiplicative chaos. The first work on this subject could
be traced back to Kolmogorov~\cite{Ko41} in 1941 regarding the local
structure of turbulence in probabiliy interpretation. The study was
developed by Yaglom~\cite{Ya66} in 1966 (introducing the cascade
structure) and Mandelbrot~\cite{Ma72,Ma74} in the early 70s (refining
the cascade structure and pointing out the necessity of using limiting
procedures). Then in 1976 Kahane and Peyri\`{e}re~\cite{KaPe76}
completed the work in~\cite{Ma74} regarding Mandelbrot measures, and
introduced several fundamental ideas for the study of multiplicative
chaos. Later in 1985 Kahane~\cite{Ka85} defined rigorously the Gaussian
multiplicative chaos suggested by Mandelbrot in~\cite{Ma72}; in
particular his theory gives a rigorous definition of the measure
$e^{\gamma h}\,dz$ where $h$ is the Gaussian free field. For more
details on this subject one can see, for example, the survey paper \cite
{BaMa04}.

Of special interest to this paper, we would like to present here more
precisely Benjamini and Schramm's result on the geometrical KPZ formula
for Mandelbrot measures: let $W$ be a positive random variable of
expectation $1/2$, and let $\{W(w)\dvtx  w\in\bigcup_{n\geq1} \{0,1\}
^n\}$ be a sequence of independent copies of $W$ encoded by the dyadic
words. The Mandelbrot measure $\mu$ on $[0,1]$ generated by $W$ is
defined as the weak limit of
\[
 \bigl(d\mu_n(x) = 2^n\cdot W(x|_1)W(x|_2)\cdots W(x|_n)\, dx
 \bigr)_{n \geq1},
\]
where for $i=1, 2,\ldots ,$ and $x\in[0,1]$, $x|_i$ stands for the first
$i$ letters of the dyadic expansion of $x$. From~\cite{KaPe76,Ka87} one
knows that if $\mathbb{E}(W\log W)<0$, then~$\mu$ is almost surely
nondegenerate and without atom, so it induces a random metric $\rho_\mu
$ on $[0,1]$ given by $\rho_\mu(x,y)=\mu([x,y])$ for $0\leq
x<y\leq1$ (such a~metric was previously considered in \cite
{Ba99}). Denote by $\dim_H$ the Hausdorff dimension with respect to the
Euclidean metric and by $\dim_H^{\rho_\mu}$ the Hausdorff dimension
with respect to $\rho_\mu$, it is shown in~\cite{BeSc09} that if
$\mathbb{E}(W^{-s})<\infty$ for all $s\in[0,1)$, then for any Borel set
$K\subset[0,1]$ with $\dim_H K=\xi_0$, almost surely $\dim_H^{\rho_\mu
} K$ is equal to a constant $\xi\in[0,1]$ satisfying
%
\begin{equation}\label{posi}
2^{-\xi_0}= \mathbb{E}(W^\xi).
\end{equation}
In the special case when $W=2^{-1}\cdot2^{-\gamma^2/4} e^{\gamma
H/2}$, where $H=\mathcal{N}(0,2\ln2)$ is a~normal random variable and
$\gamma\in[0,2)$ [notice that $\mathbb{E}(W\log W)<0$ is equivalent to
$\gamma< 2$], they obtain a geometrical KPZ formula for $\rho_\mu$,
%
\begin{equation}\label{kpzold}
\xi_0=\xi+\frac{\gamma^2}{4} \cdot\xi(1-\xi).
\end{equation}
To recover \eqref{KPZ} one may let $\Delta^0=1-\xi_0$ and $\Delta=1-\xi
$. Notice that if we consider the indefinite integral of $\mu$, that is
the function $F_\mu(x)=\mu([0,x])$ for $x\in[0,1]$, then by definition
one directly gets
\[
\dim_H^{\rho_\mu} K=\dim_H F_\mu(K).
\]
So Benjamini and Schramm's result can be also understood as a Hausdorff
dimension result for the image of the increasing process $F_\mu$.

The main goal of this paper is to extend Benjamini and Schramm's result
to signed multiplicative cascade processes, a class of random
multifractal functions recently constructed in~\cite{BaJiMa10} as a
natural generalization of~$F_\mu$. These processes are no longer
increasing functions, and their graphs normally have Hausdorff
dimension greater than $1$, so one would naturally expect a formula
that relates sets with dimension smaller than $1$ to sets with
dimension larger than $1$. This remark led us to directly consider the
case of two signed multiplicative cascades simultaneously. Before
stating in more detail the result we need to recall the definition of
two-dimensional multiplicative cascade processes. Let us begin with
some notations on the coding space.\looseness=1

\subsection*{Coding space}

Let $b\geq2$ be an integer, and let $\mathscr{A}=\{0,\ldots ,b-1\}
$ be the alphabet. Let $\mathscr{A}^*=\bigcup_{n\geq0} \mathscr
{A}^n$ (by convention $\mathscr{A}^0=\{\varnothing\}$ the set of empty
word) and $\mathscr{A}^{\mathbb{N}_+}=\{0,\ldots ,b-1\}^{\mathbb{N}_+}$.

The word obtained by concatenation of $u\in\mathscr{A}^*$ and $v\in
\mathscr{A}^*\cup\mathscr{A}^{\mathbb{N}_+}$ is denoted by $u\cdot v$
and sometimes $uv$. If $n\geq1$ and $u=u_1\cdots u_n\in\mathscr
{A}^n$, then for every $1\leq i\leq n$, the word $u_1\cdots
u_i$ is denoted by $u|_i$, and if $i=0$ then $u|_0$ stands for
$\varnothing$. Also, for any infinite word $v=v_1 v_2\cdots\in\mathscr
{A}^{\mathbb{N}_+}$ and $n\geq1$, $v|_n$ denotes the word
$v_1\cdots v_n$ and $v|_0$ the empty word.

The length of a word $w$ is denoted by $|w|=n$ if $w\in\mathscr{A}^n$
and $|w|=\infty$ if $w\in\mathscr{A}^{\mathbb{N}_+}$. Let $\pi\dvtx  w\in
\mathscr{A}^*\cup\mathscr{A}^{\mathbb{N}_+} \mapsto\sum_{i=1}^{|w|}
w_i\cdot b^{-i}$ be the canonical projection from $\mathscr{A}^*\cup
\mathscr{A}^{\mathbb{N}_+}$ onto the interval $[0,1]$. For $w\in\mathscr
{A}^*$ denote by $I_w=[\pi(w), \pi(w)+b^{-|w|})$ the $b$-adic interval
encoded by $w$.

For $x\in[0,1)$ and $n\geq1$, let $x|_n=x_1\cdots x_n$ be the
unique element of $\mathscr{A}^n$ such that $x \in I_{x_1\cdots x_n}$,
as well as $1|_n=b-1\cdots b-1$.

\subsection*{Two-dimensional multiplicative cascade processes}

Let $(\Omega,\mathcal{A},\mathbb{P})$ be the probability space, and let
$W=(W_1, W_2)$ be a random vector satisfying:
\begin{longlist}[(A2)]
\item[(A0)] $\mathbb{E}(W_1)=\mathbb{E}(W_2)=b^{-1}$;
\item[(A1)] $\exists q\in(1,2]$ such that $\mathbb
{E}(|W_1|^{q})\vee\mathbb{E}(|W_2|^q)<b^{-1}$;
\item[(A2)] $\exists s>2$ such that $\mathbb{E}(|W_1|^{-s})\vee
\mathbb{E}(|W_2|^{-s})<\infty$.
\end{longlist}

Let $\{W(w)\dvtx  w\in\mathscr{A}^*\}$ be a sequence of independent copies
of $W$.

For $k\in\{1,2\}$, $x\in[0,1]$ and $n\geq1$ define the product
\[
Q_k(x|_n)=Q_k(I_{x|_n})=W_k(x|_1)\cdot W_k(x|_2)\cdots W_k(x|_n).
\]

For $k\in\{1,2\}$ and $n\geq1$ define the random piecewise linear function
\[
F_{k,n}\dvtx  t\in[0,1]\mapsto\int_0^t b^n\cdot Q_{k}(x|_n)\, dx.
\]
From~\cite{BaJiMa10} one has almost surely $F_{k,n}$ converges
uniformly to a limit $F_k$. The two-dimensional multiplicative cascade
process considered in this paper is defined as
\[
F\dvtx  t\in[0,1] \mapsto(F_1(t), F_2(t))\in\mathbb{R}^2.
\]
Notice that if $\mathbb{P}(W_1=W_2)=1$, then almost surely $F_1=F_2$,
thus $F$ degenerates to a one-dimensional multiplicative cascade process.

\subsection*{Main result}

Given $\xi_0\in[0,1]$, denote by $\xi$ the smallest solution of the equation
%
\begin{equation}\label{xi}
b^{-\xi_0}= \mathbb{E}(|W_1|^\xi) \vee\mathbb{E}(|W_2|^\xi)
\end{equation}
and $\zeta$ the smallest solution of the equation
%
\begin{equation}\label{zeta}
b^{-\xi_0}=\mathbb{E}(|W_1|^{\zeta-1}\cdot|W_2|) \vee\mathbb
{E}(|W_1|\cdot|W_2|^{\zeta-1}).
\end{equation}
Also denote by
\[
\xi_*=-\log_b \bigl(\mathbb{E}(|W_1|) \vee\mathbb{E}(|W_2|) \bigr).
\]
From assumptions (A0) and (A1) one can easily deduce that $\xi_*\in(1/2,1]$.

\begin{theorem}\label{thm}
Let $K\subset[0,1]$ be any Borel set with $\dim_H K=\xi_0$.
\begin{longlist}[(ii)]
\item[(i)] If $\mathbb{P}(W_1= W_2)<1$, then almost surely
\[
\dim_H F(K)=\xi\wedge\zeta=
\cases{\displaystyle
\xi,  &\quad  if  $\xi_0\in[0,\xi_*] $;\cr\displaystyle
\zeta,  &\quad  if  $\xi_0\in(\xi_*, 1]$.
}\vadjust{\goodbreak}
\]
\item[(ii)] If $\mathbb{P}(W_1= W_2)=1$, then almost surely
\[
\dim_H F(K)=\xi\wedge1 =
\cases{\displaystyle
\xi,  &\quad  if  $\xi_0\in[0,\xi_*]$;  \cr\displaystyle
1,  &\quad  if  $\xi_0\in(\xi_*, 1]$.
}
\]
\end{longlist}
\end{theorem}

Let us give two examples to help understand the result.

\begin{example}\label{example1}   Let $X_1$ and $X_2$ be two random
variables both taking values $b^{-\alpha}$ and $-b^{-\alpha}$ with
respective probabilities $(1+b^{\alpha-1})/2$ and $(1-b^{\alpha-1})/2$
for some $\alpha\leq1$. Suppose that $\mathbb{P}(X_1= X_2)<1$.
Let $\gamma\geq0$ and let $H=\mathcal{N}(0,2\ln b)$ be a normal
random variable independent of $X_1$ and $X_2$. Define
\[
W_1= X_1\cdot b^{-\gamma^2/4} e^{\gamma H/2}  \quad \mbox{and} \quad  W_2=X_2\cdot
b^{-\gamma^2/4} e^{\gamma H/2}.
\]
By simple calculation one has for $\{k,l\}=\{1,2\}$,
\[
\mathbb{E}(|W_k|^\xi)=\mathbb{E}(|W_l|^\xi)=\mathbb{E} (|W_k|^{\xi
-1}\cdot|W_l| )=b^{-\xi\alpha}\cdot b^{-\xi(1-\xi)\gamma^2/4}.
\]
Then assumption (A1) [(A0) and (A2) are automatically satisfied] is
equivalent to requiring
%
\begin{equation}\label{require}
\gamma<2  \quad \mbox{and} \quad
\cases{\displaystyle
 \gamma-\gamma^2/4 < \alpha\leq1,  &\quad  if
$\gamma\geq1$;  \cr\displaystyle
\gamma^2/4+1/2<\alpha\leq1,  &\quad  if  $\gamma<1$.
}
\end{equation}
In such a case, Theorem~\ref{thm} says that for any Borel set $K\subset
[0,1]$ with $\dim_H K=\xi_0$, almost surely $\dim_H F(K)$ is equal to a
constant $\xi\in[0,2)$ satisfying
\[
\xi_0=\alpha\cdot\xi+\frac{\gamma^2}{4} \cdot\xi(1-\xi).
\]
Comparing to \eqref{kpzold}, this formula has a new parameter $\alpha$
varying in the region given by \eqref{require}, and when $\alpha<1$,
the maximal dimension $\dim_H F([0,1])$ is equal to
\[
\frac{\gamma^2+4\alpha-\sqrt{(\gamma^2+4\alpha)^2-16\gamma^2}}{2\gamma
^2} \in(1,2)
\]
if $\gamma>0$ and is equal to $1/\alpha\in(1,2)$ if $\gamma=0$.
\end{example}

\begin{example}\label{example2}   Now let
\[
W_1= X_1\cdot b^{-\gamma^2/4} e^{\gamma H/2}  \quad \mbox{and} \quad  W_2=
b^{-1}\cdot b^{-\gamma^2/4} e^{\gamma H/2} ,
\]
so $W_2$ is almost surely positive. For $\xi\geq0$ one has
\[
\mathbb{E}(|W_1|^\xi)\vee\mathbb{E}(|W_2|^\xi)=b^{-\xi\alpha}\cdot
b^{-\xi(1-\xi)\gamma^2/4},
\]
and for $\zeta\geq1$ one has
\[
\mathbb{E} (|W_1|^{\zeta-1}|W_2| )\vee\mathbb{E} (|W_2|^{\zeta
-1}|W_1| )= b^{-(\zeta-1+\alpha)}\cdot b^{-\zeta(1-\zeta)\gamma^2/4}
\]
as well as $\xi_*=\alpha$. We need the same condition as in \eqref
{require}. In this case, since $F_2$ is almost surely increasing, one
can deduce a random metric $\rho_F$ from $F$ on $[0,1]$ given by $\rho
_F(x,y)=|F(x)-F(y)|$ for $x,y\in[0,1]$. Then Theorem~\ref{thm} says
that for any Borel set $K\subset[0,1]$ with $\dim_H K=\xi_0$, almost
surely $\dim_H^{\rho_F} K$ is equal to a constant $\xi\in[0,2)$ satisfying
\[
\cases{\displaystyle
  \xi_0=\alpha\cdot\xi+\frac{\gamma^2}{4} \cdot\xi(1-\xi)   ,&\quad
 if  $\xi_0\in[0,\alpha]$; \cr\displaystyle
  \xi_0=\xi-1+\alpha+\frac{\gamma^2}{4}\cdot\xi(1-\xi)   ,&\quad  if
$\xi_0\in(\alpha,1]$.
}
\]
If $\alpha=1$, then we go back to \eqref{kpzold}. If $\alpha<1$, then
this KPZ-type formula has a phase transition at $\alpha$, and the
maximal dimension $\dim_H^{\rho_F}[0,1]$ is equal to
\[
\frac{\gamma^2+4-\sqrt{(\gamma^2+4)^2-16\gamma^2(2-\alpha)}}{2\gamma^2}
\in(1,2)
\]
if $\gamma>0$ and is equal to $2-\alpha\in(1,3/2)$ if $\gamma=0$.
\end{example}

\begin{remark}\label{reason}
The reason why we consider the two-dimensional case can be easily seen
from Theorem~\ref{thm} and Examples~\ref{example1},~\ref{example2}. If
we only consider the one-dimensional case, as already shown in Theorem
\ref{thm}(ii), the formula will also have a phase transition at $\xi
_*$, but such a phase transition is indeed caused by the limitation of
the image space.
\end{remark}

\begin{remark}
Examples~\ref{example1} and~\ref{example2} are special cases of Theorem
\ref{thm}. In general, the theorem could provide us with more colorful
formulas. In principle, the formula can have as many points of phase
transition as we want.
\end{remark}

\begin{remark}
Finding the Hausdorff dimension of the image of a stochastic process
restricted to any Borel set is a classical problem in probability
theory. The first work on this subject could be traced back to L\'{e}vy
\cite{Levy53} and Taylor~\cite{Taylor53} in 1953, regarding the
Hausdorff dimension and Hausdorff measure of the image of Brownian
motion. Since then much progress has been made for fractional Brownian
motion, stable L\'{e}vy process and many other processes. We refer to
the survey paper~\cite{Xiao04} and the references therein for more
information on this subject.
\end{remark}

The proof of Theorem~\ref{thm} will be given in the next section. We
end  this section with some preliminaries.

\subsection*{Hausdorff dimension}

If $(X,\rho)$ is a locally compact metric space, for $d\geq0$,
$\delta>0$ and $K\subset X$ let
\[
\mathcal{H}^{\rho,d}_\delta(K)=\inf \biggl\{\sum_{i\in I} |U_i|_\rho^d
\biggr\},
\]
where the infimum is taken over the set of all the at most countable
coverings $\{U_i\}_{i\in I}$ of $K$ such that $0\leq|U_i|_\rho
\leq\delta$, where $|U_i|_\rho$ stands for the diameter of~$U_i$
with respect to $\rho$. Define
\[
\mathcal{H}^{\rho,d}(K)=\lim_{\delta\searrow0} \mathcal{H}^{\rho
,d}_\delta(K).
\]
Then $\mathcal{H}^{\rho,d}(K)$ is called the $d$-dimensional Hausdorff
measure of $K$ with respect to $\rho$, and the Hausdorff dimension of
$K$ with respect to $\rho$ is the number
\[
\dim_H^\rho K=\inf\{d\dvtx \mathcal{H}^{\rho,d}(K)<\infty\}.
\]
When $\rho$ is the standard Euclidean metric, we often omit the index
$\rho$.

\subsection*{Stationary self-similarity of multiplicative cascade processes}

$\!\!\!$For $k\in\{1,2\}$, $w\in\mathscr{A}^*$ and $n\geq1$ define
\[
F_{k,n}^{[w]}\dvtx  t\in[0,1]\mapsto\int_0^t b^n\cdot W_{k}(w\cdot
x|_1)\cdots W_k(w\cdot x|_n)\, dx.
\]
Since $\mathscr{A}^*$ is countable, we have almost surely for\vspace*{-2pt} all $w\in
\mathscr{A}^*$, $F_{k,n}^{[w]}$ converges uniformly to a limit
$F_k^{[w]}$ and $F_k^{[w]}$ has the same law as $F_k$.

By construction for any $w\in\mathscr{A}^*$ and $t\in[0,1]$ one has
%
\begin{equation}\label{selfsimilar}
F_k\bigl(\pi(w)+t\cdot b^{-|w|}\bigr)-F_k(\pi(w))=Q_k(w)\cdot F^{[w]}_k(t).
\end{equation}

For $w\in\mathscr{A}^*$ define
\[
Z_k(w)=F_k^{[w]}(1)
\]
and
\[
X_k(w)=\sup_{s,t\in[0,1]} \bigl|F^{[w]}_k(s)-F^{[w]}_k(t)\bigr|.
\]
Then from \eqref{selfsimilar} one has
\[
F_k\bigl(\pi(w)+b^{-|w|}\bigr)-F_k(\pi(w))=Q_k(w)\cdot Z_k(w)
\]
and
\begin{eqnarray*}
O_k(w)=O_k(I_w):=\sup_{s,t\in I_w} |F_k(s)-F_k(t)| = |Q_k(w)| \cdot X_k(w),
\end{eqnarray*}
where $Q_k(w)$ is independent of $Z_k(w)$ and $X_k(w)$.

We will use the convention that $Z_k=Z_k(\varnothing)$ and
$X_k=X_k(\varnothing)$.

By direct calculation, for any $q_1,q_2\in\mathbb{R}$ and $w\in\mathscr
{A}_*$ one has
\[
\mathbb{E} (O_1(w)^{q_1}O_2(w)^{q_2} ) = \mathbb{E}
(|W_1|^{q_1}|W_2|^{q_2} )^{|w|} \cdot\mathbb{E}(X_1^{q_1}X_2^{q_2}),
\]
whenever the expectation exists.

\subsection*{Moments control}

It is proved in~\cite{BaJiMa10} that for $k\in\{1,2\}$:
%
\begin{eqnarray}\label{momcon}
\mbox{(i)}&&\hspace*{6pt} \mbox{If } \mathbb{E}(|W_k|^q)<b^{-1} \mbox{ for some }
q>1, \mbox{ then } \mathbb{E}(X_k^q)<\infty; \nonumber
\\[-8pt]
\\[-8pt]
\mbox{(ii)}&&\hspace*{6pt}  \mbox{If } \mathbb{E}(|W_k|^{-s})<\infty\mbox{ for some } s>0,
\mbox{ then } \mathbb{E}(X_k^{-s})<\infty.
\nonumber
\end{eqnarray}

\section{\texorpdfstring{Proof of Theorem \protect\ref{thm}}{Proof of Theorem 1}}\label{proof}

\subsection{Upper bound estimate}\label{uppergene}

For $p\geq0$ let
\[
\phi(p) = \mathbb{E} (|W_{1}|^p ) \vee\mathbb{E}
(|W_{2}|^p )
\]
and
\[
\widetilde{\phi}(p) = \mathbb{E} (|W_{1}|^{p-1}\cdot|W_2| ) \vee
\mathbb{E} (|W_1|\cdot|W_{2}|^{p-1} ).
\]

We have the following lemma:
\begin{lemma}\label{lemma1}
One has $\phi(p) \leq\widetilde{\phi}(p)$ if $p\in[0,1]$ and $\phi
(p)\geq\widetilde{\phi}(p)$ if $p \geq1$.
\end{lemma}

\begin{pf}
Obviously $\phi(1)=\widetilde{\phi}(1)$.

Since $\frac{|W_1|}{|W_2|}+\frac{|W_2|}{|W_1|}\geq2$, we get\vspace*{2pt}
$\widetilde{\phi}(0)\geq1 = \phi(0)$.

Let $\{k,l\}=\{1,2\}$. For $p>1$ from H\"{o}lder's inequality one gets
\begin{eqnarray*}
\mathbb{E} (|W_{k}|^{p-1}|W_{l}| ) &\leq& \mathbb{E}
\bigl(|W_{k}|^{(p-1)\cdot\fracc{p}{p-1}} \bigr)^{\fracb{p-1}{p}} \cdot\mathbb
{E} (|W_{l}|^p )^{\fraca{1}{p}} \\
&=& \mathbb{E} (|W_{k}|^{p} )^{\fracb{p-1}{p}} \cdot\mathbb{E}
(|W_{l}|^p )^{\fraca{1}{p}} \\
&\leq& \phi(p).
\end{eqnarray*}
This implies $\widetilde{\phi}(p)\leq\phi(p)$. For $p\in(0,1)$
from H\"{o}lder's inequality one gets
\begin{eqnarray*}
\mathbb{E} (|W_{k}|^{p} ) &=& \mathbb{E}\bigl (|W_{k}|^{p} \cdot
|W_l|^{-p(1-p)} \cdot|W_l|^{p(1-p)}  \bigr) \\
& \leq& \mathbb{E} \bigl( \bigl(|W_{k}|^p \cdot|W_l|^{-p(1-p)}
\bigr)^{\fraca{1}{p}}  \bigr)^p\cdot\mathbb{E} \bigl(  \bigl(|W_l|^{p(1-p)}
\bigr)^{\fracc{1}{1-p}}  \bigr)^{1-p}\\
& = & \mathbb{E} (|W_k|\cdot|W_l|^{p-1} )^p\cdot\mathbb{E}
(|W_l|^p )^{1-p}\\
& \leq& \widetilde{\phi}(p)^p \cdot\mathbb{E} (|W_l|^p )^{1-p}.
\end{eqnarray*}
In an analogous way one can also obtain
\[
\mathbb{E} (|W_{l}|^p ) \leq\widetilde{\phi}(p)^p\cdot
\mathbb{E} (|W_k|^p )^{1-p}.
\]
Then
\[
\mathbb{E} (|W_{k}|^p ) \leq\widetilde{\phi}(p)^p\cdot
\widetilde{\phi}(p)^{p(1-p)} \cdot\mathbb{E} (|W_k|^p )^{(1-p)(1-p)},
\]
which implies $\mathbb{E} (|W_{k}|^p )\leq\widetilde{\phi
}(p)$, thus $\phi(p) \leq\widetilde{\phi}(p)$.
\end{pf}

Given $\xi_0\in[0,1]$ recall the definition of $\xi$ and $\zeta$ in
\eqref{xi} and \eqref{zeta}.

Notice that under assumptions (A0) and (A1), Lemma~\ref{lemma1} ensures
that, by the convexity of $\phi$\vadjust{\goodbreak} and $\widetilde{\phi}$, $\phi$ and
$\widetilde{\phi}$ are non increasing on $\widetilde{\phi
}^{-1}([0,1])$. This implies that $\xi\leq\zeta\leq1$ if
$\xi_0\in[0,\xi_*]$ and $\xi\geq\zeta> 1$ if $\xi_0\in(\xi
_*,1]$, as well as $\phi'(\xi+) \leq0$ and $\widetilde{\phi
}'(\zeta+)\leq0$. Thus given any $\epsilon>0$ small enough, one
can find an $\eta>0$ such that
\[
\phi(\xi+\eta) \leq b^{-(\xi_0+\epsilon)}  \quad \mbox{and} \quad  \widetilde
{\phi}(\zeta+\eta) \leq b^{-(\xi_0+\epsilon)}.
\]
From the moments control \eqref{momcon} it is easy to deduce that
\[
\mathbb{E}(X_1^{\xi+\eta})\vee\mathbb{E}(X_2^{\xi+\eta}) <\infty
\]
as well as for $\{k,l\}=\{1,2\}$ and $\zeta>1$,
\[
\mathbb{E}(X_k^{\zeta+\eta-1}X_l) \leq\mathbb{E}(X_k^{\zeta+\eta
})^{\frace{\zeta+\eta-1}{\zeta+\eta}}\cdot\mathbb{E}(X_l^{\zeta+\eta
})^{\fracc{1}{\zeta+\eta}} <\infty.
\]

From the definition of Hausdorff dimension, for each $n\geq1$ one
can find a~sequence $\mathcal{I}_n$ of $b$-adic intervals such that
\[
K\subset\bigcup_{I\in\mathcal{I}_n} I  \quad \mbox{and} \quad  \sum_{I\in\mathcal
{I}_n} |I|^{\xi_0+\epsilon} \leq2^{-n}.
\]
Let $\delta_n=\sup_{I\in\mathcal{I}_n} |F(I)|$. Since $F$ is almost
surely continuous, $\delta_n \to0$ almost surely. For any interval
$I\in\mathcal{I}_n$ denote by
\[
O_*(I)= O_1(I) \wedge O_2(I)  \quad \mbox{and} \quad  O^*(I)=O_{1}(I) \vee O_2(I).
\]
Then we can obtain the desired upper bounds from the following two facts:
\begin{longlist}[(ii)]
\item[(i)] If $\xi_0\in[0,\xi_*]$: for each $I\subset\mathcal{I}_n$
one can use a single square of side length $2O^*(I)$ to cover $F(I)$, thus
\begin{eqnarray*}
\mathbb{E} (\mathcal{H}^{\xi+\eta}_{\delta_n}( F(K)))
& \leq& 2^{\xi+\eta}\mathbb{E} \biggl(\sum_{I\in\mathcal{I}_n}
O_1(I)^{\xi+\eta} \vee O_2(I)^{\xi+\eta}  \biggr) \\
& \leq& 2^{\xi+\eta}\sum_{I\in\mathcal{I}_n} \mathbb{E}\bigl(
O_1(I)^{\xi+\eta}+O_2(I)^{\xi+\eta}\bigr)\\
& \leq& C \cdot\sum_{I\in\mathcal{I}_n} |I|^{\xi_0+\epsilon} \\
&\leq& C \cdot2^{-n},
\end{eqnarray*}
where $C=2^{\xi+\eta+1} \mathbb{E}(X_1^{\xi+\eta}) \vee\mathbb
{E}(X_2^{\xi+\eta})$.
\item[(ii)] If $\xi_0\in(\xi_*,1]$: for each $I\subset\mathcal{I}_n$
one can use no more than $\lfloor O^*(I)/O_*(I) \rfloor$-many squares
of side length $2O_*(I)$ to cover $F(I)$, thus
\begin{eqnarray*}
&&\mathbb{E} (\mathcal{H}^{\zeta+\eta}_{\delta_n}( F(K)) ) \\
&& \qquad \leq 2^{\zeta+\eta} \mathbb{E} \biggl(\sum_{I\in\mathcal{I}_n}
\biggl(\frac{O_2(I)}{O_1(I)}\cdot O_1(I)^{\zeta+\eta} \biggr)\vee \biggl(\frac
{O_1(I)}{O_2(I)}\cdot O_2(I)^{\zeta+\eta} \biggr)  \biggr) \\
&& \qquad \leq 2^{\zeta+\eta} \sum_{I\in\mathcal{I}_n} \mathbb{E}\bigl(
O_2(I)O_1(I)^{\zeta+\eta-1} +O_1(I)O_2(I)^{\zeta+\eta-1}\bigr) \\
&& \qquad \leq C' \cdot\sum_{I\in\mathcal{I}_n} |I|^{\xi_0+\epsilon} \\
&& \qquad \leq C' \cdot2^{-n},
\end{eqnarray*}
where $C'=2^{\zeta+\eta+1} \mathbb{E}(X_1^{\zeta+\eta-1}X_2) \vee
\mathbb{E}(X_2^{\zeta+\eta-1}X_1) $.
\end{longlist}

\subsection{Lower bound estimate}\label{lb}

We will use a similar method as in~\cite{BeSc09} to estimate the lower
bound. First we consider the case $\mathbb{P}(W_1= W_2)<1$.

There is nothing to prove when $\dim_H K=0$, since $F(K)$ is always
nonempty. Let $\dim_H K=\xi_0>0$. Given any $\delta\in(0,\xi_0)$, due
to Frostman's lemma there exists a Borel probability measure $\mu_0$
carried by $K$ such that
\[
\iint_{s,t\in[0,1]} \frac{d\mu_0(s)\,d\mu_0(t)}{|s-t|^{\xi
_0-\delta}} <\infty.
\]

Let $\{k,l\}=\{1,2\}$ and let $d\in(0,2)$ be the unique number such that
\[
\cases{\displaystyle
\mathbb{E}(|W_{k}|^{d})=b^{-(\xi_0-\delta)},  &\quad  if  $\xi_0\in(0,\xi
_*]$; \vspace*{2pt}\cr\displaystyle
\mathbb{E}(|W_{k}|\cdot|W_l|^{d-1})=b^{-(\xi_0-\delta)},  &\quad  if  $\xi
_0\in(\xi_*,1]$.
}
\]
We may assume that $\delta$ is small enough such that $d>1$ if $\xi
_0\in(\xi_*,1]$, and $d\in(0,1)$ if $\xi_0\in(0,\xi_*]$.

For $w\in\mathscr{A}^*$ let
\[
\widetilde{W}(w)=
\cases{\displaystyle
b^{\xi_0-\delta}\cdot|W_k(w)|^{d},  &\quad  if  $\xi_0\in(0,\xi_*]$; \vspace*{2pt}\cr\displaystyle
b^{\xi_0-\delta}\cdot|W_{k}(w)|\cdot|W_l(w)|^{d-1},  &\quad  if  $\xi
_0\in(\xi_*,1]$
}
\]
and
\[
Q(w)=\widetilde{W}(w|_1) \widetilde{W}(w|_2)\cdots\widetilde{W}(w).
\]

For $n\geq1$ define the random measure $\mu_n$ by
\[
d\mu_n(x)=Q(x|_n) \, d\mu_0(x).
\]
By construction, $(\mu_n)_{n\geq1}$ is a measure-valued
martingale thus yields a weak limit~$\mu$, and $\mu([0,1]\setminus
K)=0$ almost surely.

For $s,t\in[0,1]$ define
%
\begin{equation}\label{kgammale1}
\mathcal{K}_n^d(s,t)= |F_k(s)-F_k(t)|^{d} \vee O_k(s|_n)^{d}
\end{equation}
if $d\in(0,1]$ and
%
\begin{eqnarray}\label{kgammage1}
 \mathcal{K}_n^d(s,t)  &=& \bigl(|F_k(s)-F_k(t)|^2+|F_l(s)-F_l(t)|^2
\bigr)^{\fraca{d}{2}}\nonumber
\\[-8pt]
\\[-8pt]
 &&{}\vee \bigl(O_k(s|_n)^2+O_l(s|_n)^2 \bigr)^{\fraca{d}{2}}
\nonumber
\end{eqnarray}
if $d>1$. Due to the continuity of $F$, one has almost surely $\mathcal
{K}_n^d$ converges uniformly to
\[
\mathcal{K}^d(s,t)=
\cases{\displaystyle
|F_k(s)-F_k(t)|^{d}  ,&\quad  if  $d\in(0,1]$; \vspace*{2pt}\cr\displaystyle
 \bigl(|F_k(s)-F_k(t)|^2+|F_l(s)-F_l(t)|^2 \bigr)^{\fraca{d}{2}}  ,&\quad
 if   $d >1$.
}
\]

We have the following proposition:
\begin{proposition}\label{prop1}
There exists a constant $C$ such that for any $0\leq s <
t\leq1$ and $n\geq1$,
\[
\mathbb{E} \biggl(\frac{d\mu_n(s)\,d\mu_n(t)}{\mathcal{K}_n^d(s,t)}
 \biggr)\leq C \cdot\frac{ d\mu_0(s)\,d\mu_0(t) }{|s-t|^{\xi
_0-\delta}}.
\]
\end{proposition}

By using Fubini's theorem, Proposition~\ref{prop1} yields that for any
$n\geq1$,
%
\begin{equation}\label{finite}
\mathbb{E} \biggl(\iint_{s,t\in[0,1]} \frac{ d\mu_n(s)\,d\mu_n(t)}{
\mathcal{K}_n^d(s,t)}  \biggr) \leq2C \iint_{s,t\in[0,1]} \frac
{d\mu_0(s)\,d\mu_0(t)}{|s-t|^{\xi_0-\delta}} <\infty.
\end{equation}

For any $s,t\in[0,1]$ one has
\[
\mathcal{K}_n^d(s,t) \leq\sup_{s,t\in[0,1]} |F_k(s)-F_k(t)|^d =X_k^d,
\]
so \eqref{finite} implies
%
\begin{equation}\label{exinfty}
\sup_{n\geq1} \mathbb{E}\bigl(X_k^{-d}\cdot\mu_n([0,1])^2\bigr) <\infty.
\end{equation}
Notice that for $d\in(0,1)$ we have
\begin{eqnarray*}
\mathbb{E}\bigl(\mu_n([0,1])^{\fracc{2}{1+d}}\bigr)
& = & \mathbb{E}\bigl(X_k^{\fracc{d}{1+d}}\cdot X_k^{-\fracc{d}{1+d}} \cdot\mu
_n([0,1])^{\fracc{2}{1+d}}\bigr) \\
& \leq& \mathbb{E}(X_k)^{\fracc{d}{1+d}} \cdot\mathbb
{E}\bigl(X_k^{-d}\cdot\mu_n([0,1])^2\bigr)^{\fracc{1}{1+d}},
\end{eqnarray*}
and for $d\in(1,2)$ we have for any $\epsilon>0$,
\begin{eqnarray*}
\mathbb{E}\big(\mu_n([0,1])^{1+\epsilon}\big)&=&
\mathbb{E}\big(X_k^{d(1+\epsilon)/2}\cdot X_k^{-d(1+\epsilon)/2} \cdot \mu_n([0,1])^{1+\epsilon}\big) \\
&\leq& \mathbb{E}\big(X_k^{d(1+\epsilon)/(1-\epsilon)}\big)^{(1-\epsilon)/2} \cdot \mathbb{E}\big(X_k^{-d}\cdot \mu_n([0,1])^2\big)^{(1+\epsilon)/2}.
\end{eqnarray*}
Thus by using the corresponding martingale convergence theorem we get
from \eqref{exinfty} that $\mathbb{E}(\mu([0,1]))=1$. Then by using the
same tail event argument as in~\cite{BeSc09} we can get $\mathbb{P}(\mu
([0,1])>0)=1$.

Due to the fact that almost surely $\mu_n$ converges weakly to $\mu$
and $\mathcal{K}^d_n$ converges uniformly to $\mathcal{K}^d$, we get
from \eqref{finite} that
\[
\mathbb{E} \biggl(\iint_{s,t\in[0,1]} \frac{d\mu(s)\,d\mu
(t)}{\mathcal{K}^d(s,t)} \biggr) \leq\liminf_{n\to\infty}\mathbb
{E} \biggl(\iint_{s,t\in[0,1]} \frac{ d\mu_n(s)\,d\mu_n(t)
}{\mathcal{K}_n^d(s,t)} \biggr) <\infty.
\]
Since almost surely $\mu$ is carried by $K$ and $\mu(K)>0$, by using
the mass distribution principle we get the desired lower bound.\vadjust{\goodbreak}

For the case $\mathbb{P}(W_1= W_2)=1$, it is the same proof as above
when $\xi_0\in(0,\xi_*]$. When $\xi_0\in(\xi_*,1]$, we may take $d\in
(0,1)$ such that $\mathbb{E}(|W_k|^{d})=b^{-(\xi_*-\delta)}$ and for
$w\in\mathscr{A}_*$ let
\[
\widetilde{W}(w)=b^{\xi_*-\delta}\cdot|W_k(w)|^{d}.
\]
Then the same procedure as the case $\xi_0\in(0,\xi_*]$ will yield a
lower bound $d$, which can be arbitrarily close to $1$, thus the
conclusion.\vspace*{-2pt}

\subsection{\texorpdfstring{Proof of Proposition \protect\ref{prop1}}{Proof of Proposition 1}}

Recall that $Z_k=F_k(1)$. We will frequently use the following lemma,
whose proof will be given in the Section~\ref{prooflemma2}.\vspace*{-2pt}

\begin{lemma}\label{condiexpec2}
{\smallskipamount=0pt
\begin{longlist}[(ii)]
\item[(i)] For any $d\in(0,1)$ there exists a constant $C_d$ such that
for any constants $A,B\in\mathbb{R}$ with $A\neq0$, one has
\[
\mathbb{E}(|A Z_k+B|^{-d}) \leq C_d \cdot|A|^{-d}.
\]
\item[(ii)] If $\mathbb{P}(W_1= W_2)<1$, then for any $d\in(1,2)$ there
exists a constant $C_d$ such that for any constants $A_1,A_2,B_1,B_2
\in\mathbb{R}$ with $A_1A_2\neq0$, one has
\[
\mathbb{E} \bigl( (|A_1 Z_k+B_1|^2+|A_2 Z_l+B_2|^2)^{-d/2} \bigr) \leq
C_d \cdot|A_1|^{-1}\cdot|A_2|^{-d+1}.\vspace*{-2pt}
\]
\end{longlist}}
\end{lemma}

For $n\geq1$ and $w\in\mathscr{A}^n\setminus\{b-1\cdots b-1\}$
denote by $w^+$ the unique word in~$\mathscr{A}^n$ such that $\pi
(w^+)=\pi(w)+b^{-n}$.

Since $s<t$, there exists a unique $j\geq0$ such that
$s|_j^+=t|_j$ and $s|_{j+1}^+\neq t|_{j+1}$. This implies $\pi
(s|_{j+1}^{+})+b^{-j-1}\leq t$ and
\[
b^{-(j+1)}\leq|s-t| \leq2b^{-j} \leq b^{-(j-1)}.
\]
Notice that one has either $s_{j+1}\in\{0,\ldots , b-2\}$ or
$s_{j+1}=b-1$. Without loss of generality we may assume $s_{j+1}\in\{
0,\ldots , b-2\}$ thus $s|_{j+1}^+=s|_j \cdot r$ for $r= s_{j+1}+1\in\{
1,\ldots , b-1\}$.

Recall the definition of $\mathcal{K}_n^d$ in \eqref{kgammale1} and
\eqref{kgammage1}. We have the following two situations.\vspace*{-2pt}

\subsubsection{When $d<1$} \label{gamma<1}

\mbox{}

(i) If $j\geq n$, then
\begin{eqnarray*}
\frac{d\mu_n(s)\,d\mu_n(t)}{\mathcal{K}_n^d(s,t)}
&\leq&
O_k(s|_n)^{-d}\cdot Q(s|_n) \cdot Q(t|_n)\, d\mu_0(s)\,d\mu_0(t)\\
&=& b^{n(\xi_0-\delta)}\cdot X_k(s|_n)^{-d} \cdot|Q(t|_n)| \, d\mu
_0(s)\,d\mu_0(t).
\end{eqnarray*}
Since $X_k(s|_n)$ and $Q(t|_n)$ are independent, we get
\begin{eqnarray*}
\mathbb{E} \biggl(\frac{d\mu_n(s)\,d\mu_n(t) }{\mathcal
{K}_n^d(s,t)} \biggr)
&\leq& b^{n(\xi_0-\delta)}\cdot\mathbb
{E}(X_k^{-d})\, d\mu_0(s)\,d\mu_0(t) \\
&\leq& b^{\xi_0-\delta}\cdot\mathbb{E}(X_k^{-d}) \cdot b^{(j-1)(\xi
_0-\delta)}\, d\mu_0(s)\,d\mu_0(t)\\
&\leq& b^{\xi_0-\delta}\cdot\mathbb{E}(X_k^{-d}) \cdot\frac{d\mu_0(s)\,d\mu_0(t)}{|s-t|^{\xi_0-\delta}}.
\end{eqnarray*}

(ii) If $j\leq n-1$, then
\begin{eqnarray*}
\mathcal{K}_n^d(s,t)^{-1}
&\leq& |F_k(s)-F_k(t)|^{-d} \\
&=&\bigl|Q_k(s|_{j+1}^+)\cdot Z_k(s|_{j+1}^+)+\Delta_k\bigr|^{-d},
\end{eqnarray*}
where $\Delta_k=F_k(t)-F_k(\pi
(s|_{j+1}^{+})+b^{-j-1})+F_k(s|_{j+1}^+)-F_k(s)$. Notice that
$Z_k(s|_{j+1}^+)$ is independent of $Q(s|_n)$, $Q(t|_n)$,
$Q_k(s|_{j+1}^+)$ and $\Delta_k$. Let
%
\begin{equation}\label{sigmasj}
\mathcal{A}(s|_{j+1}^+)=\sigma \bigl(W(w)\dvtx  |w|\leq j+1
\mbox{ or }
w|_{j+1}\neq s|_{j+1}^+ \bigr).
\end{equation}
From Lemma~\ref{condiexpec2}(i) we get
\begin{eqnarray*}
&&\mathbb{E} \biggl(\frac{d\mu_n(s)\,d\mu_n(t)}{\mathcal
{K}_n^d(s,t)} \Bigm| \mathcal{A}(s|_{j+1}^+) \biggr) \\
&& \qquad \leq C_d \cdot\bigl|Q_k(s|_{j}\cdot r)\bigr|^{-d} \cdot Q(s|_n) \cdot
Q(t|_n)\, d\mu_0(s)\,d\mu_0(t)\\
&& \qquad = C_d \cdot\bigl|W_k(s|_j\cdot r)\bigr|^{-d} \cdot b^{(j+1)(\xi_0-\delta
)}\cdot\prod_{i=j+1}^n \widetilde{W}(s|_i) \cdot Q(t|_n) \, d\mu
_0(s)\,d\mu_0(t).
\end{eqnarray*}
Since all the random variables in the above products are independent,
we get
\begin{eqnarray*}
\mathbb{E} \biggl(\frac{d\mu_n(s)\,d\mu_n(t)}{\mathcal{K}_n^d(s,t)}
 \biggr) &\leq& C_d\cdot\mathbb{E}(|W_k|^{-d}) \cdot b^{(j+1)(\xi
_0-\delta)} \, d\mu_0(s)\,d\mu_0(t) \\
&\leq& C_d\cdot b^{2(\xi_0-\delta)}\cdot\mathbb{E}(|W_k|^{-d})
\cdot\frac{ d\mu_0(s)\,d\mu_0(t)}{|s-t|^{\xi_0-\delta}}.
\end{eqnarray*}

\subsubsection{When $d>1$}\label{gamma>1}

\mbox{}

(i) If $j\geq n$, then
\begin{eqnarray*}
\frac{ d\mu_n(s)\,d\mu_n(t)}{\mathcal{K}_n^d(s,t)} &\leq&
\frac{Q(s|_n) \cdot Q(t|_n) \, d\mu_0(s)\,d\mu
_0(t)}{(O_k(s|_n)^2+O_l(s|_n)^2)^{d/2}} \\
&\leq& \frac{Q(s|_n) \cdot Q(t|_n) \, d\mu_0(s)\,d\mu
_0(t)}{((Q_k(s|_n)\cdot Z_k(s|_n))^2+(Q_l(s|_n)\cdot Z_l(s|_n))^2)^{d/2}}.
\end{eqnarray*}
Let $\mathcal{A}_n=\sigma(W(w)\dvtx  |w| \leq n)$. From Lemma \ref
{condiexpec2}(ii) we get
\begin{eqnarray*}
&&\mathbb{E} \biggl( \frac{d\mu_n(s)\,d\mu_n(t)}{\mathcal
{K}_n^d(s,t)} \Bigm| \mathcal{A}_n  \biggr) \\
&& \qquad \leq C_d\cdot\bigl(|Q_k(s|_n)|\cdot|Q_l(s|_n)|^{d-1}\bigr)^{-1} \cdot
Q(s|_n) \cdot Q(t|_n) \, d\mu_0(s)\,d\mu_0(t) \\
&& \qquad =C_d\cdot b^{n(\xi_0-\delta)} \cdot Q(t|_n)\, d\mu_0(s)\,d\mu_0(t).
\end{eqnarray*}
This implies
\begin{eqnarray*}
\mathbb{E} \biggl(\frac{d\mu_n(s)\,d\mu_n(t)}{\mathcal{K}_n^d(s,t)}
 \biggr)
& \leq& C_d\cdot b^{n(\xi_0-\delta)}\, d\mu_0(s)\,d\mu_0(t)
\\
& \leq& C_d\cdot b^{\xi_0-\delta} \cdot\frac{ d\mu_0(s)\,d\mu_0(t)}{|s-t|^{\xi_0-\delta}}.
\end{eqnarray*}

(ii) If $j\leq n-1$, like in Section~\ref{gamma<1}(ii) one has
\begin{eqnarray*}
&&\mathcal{K}_n^d(s,t)^{-1} \\
& & \qquad \leq \bigl(|F_k(s)-F_k(t)|^{2}+|F_l(s)-F_l(t)|^{2}\bigr)^{-d/2} \\
&& \qquad  =  \bigl( \bigl|Q_k(s|_{j+1}^+)\cdot Z_k(s|_{j+1}^+)+\Delta
_k\bigr|^{2}+\bigl|Q_l(s|_{j+1}^+)\cdot Z_l(s|_{j+1}^+)+\Delta_l\bigr|^{2}\bigr)^{-d/2}.
\end{eqnarray*}
By using Lemma~\ref{condiexpec2}(ii) we get
\begin{eqnarray*}
&&\mathbb{E}\biggl (\frac{d\mu_n(s)\,d\mu_n(t)}{\mathcal
{K}_n^d(s,t)} \Bigm| \mathcal{A}(s|_{j+1}^+) \biggr) \\
&& \qquad \leq C_d \cdot\bigl(\bigl|Q_k(s|_{j}\cdot r)\bigr|\cdot\bigl|Q_l(s|_{j}\cdot
r)\bigr|^{d-1}\bigr)^{-1} \cdot Q(s|_n) \cdot Q(t|_n)\, d\mu_0(s)\,d\mu
_0(t) \\
&& \qquad = C_d \cdot Q(s|_{j}\cdot r)^{-1} \cdot Q(s|_n) \cdot Q(t|_n) \, d\mu_0(s)\,d\mu_0(t) \\
&& \qquad = C_d \cdot\widetilde{W}(s|_{j}\cdot r)^{-1} \cdot b^{(j+1)(\xi
_0-\delta)}\cdot\prod_{i=j+1}^{n} \widetilde{W}(s|_i) \cdot Q(t|_n) \,
d\mu_0(s)\,d\mu_0(t) .
\end{eqnarray*}
All the random variables in the above products are independent, so
\begin{eqnarray*}
\mathbb{E} \biggl(\frac{d\mu_n(s)\,d\mu_n(t)}{\mathcal{K}_n^d(s,t)}
 \biggr) &\leq& C_d \cdot\mathbb{E}(|W_k|^{-1}|W_l|^{1-d}) \cdot
b^{(j+1)(\xi_0-\delta)}\, d\mu_0(s)\,d\mu_0(t) \\
&\leq& C_d\cdot\mathbb{E}(|W_k|^{-1}|W_l|^{1-d}) b^{2(\xi_0-\delta
)} \cdot \frac{ d\mu_0(s)\,d\mu_0(t)}{|s-t|^{\xi_0-\delta}}.
\end{eqnarray*}

\subsubsection{Conclusion}
Let
\[
C=
\cases{\displaystyle
\max \bigl\{ b^{\xi_0-\delta} \mathbb{E}(X_k^{-d}), C_d b^{2(\xi_0-\delta
)}\mathbb{E}(|W_k|^{-d}) \bigr\},  &\quad  if  $\xi_0\in(0,\xi_*]$; \vspace*{2pt}\cr\displaystyle
\max \bigl\{ C_d b^{\xi_0-\delta}, C_d b^{2(\xi_0-\delta)} \mathbb
{E}(|W_k|^{-1}|W_l|^{1-d}) \bigr\},  &\quad  if  $\xi_0\in(\xi_*,1]$.
}
\]
Then we get the conclusion from Section~\ref{gamma<1} and~\ref{gamma>1}.

\subsection{\texorpdfstring{Proof of Lemma \protect\ref{condiexpec2}}{Proof of Lemma 2}}\label{prooflemma2}

\mbox{}

(i) Let $\varphi_k(x)=\mathbb{E}(e^{ixZ_k})$ be the characteristic
function of $Z_k$. From \eqref{selfsimilar} we have the following
functional equation:
%
\begin{equation}\label{WZ}
Z_k=\sum_{j=0}^{b-1}W_k(j)\cdot Z_k(j).
\end{equation}
This implies
\[
\varphi_k(x)=\mathbb{E} \Biggl(\prod_{j=0}^{b-1} \varphi_k\bigl(x\cdot W_k(j)\bigr)
 \Biggr).
\]
Notice that given $x\in\mathbb{R}$ one has $|\varphi_k(x)|=|\varphi
_k(-x)|=|\varphi_k(|x|)|$, so
\begin{eqnarray*}
|\varphi_k(x)| &\leq& \mathbb{E} \Biggl(\prod_{j=0}^{b-1}  \bigl|\varphi
_k\bigl(x\cdot W_k(j)\bigr)  \bigr| \Biggr) \\
&=& \mathbb{E}\Biggl (\prod_{j=0}^{b-1}  \bigl| \varphi_k\bigl(x\cdot
|W_k(j)|\bigr) \bigr|  \Biggr),
\end{eqnarray*}
which implies
%
\begin{equation}\label{varphik}
|\varphi_k(x)| \leq\mathbb{E}\bigl (  \bigl|\varphi_k(x\cdot
|W_k|) \bigr| \bigr)^b.
\end{equation}
Starting from \eqref{varphik} and following the proof of Theorem 2.1 in
\cite{Liu01} one can prove the following result:
\[
\mbox{If } \mathbb{E}(|W_k|^{-s})<\infty\mbox{ for some } s>0, \mbox{ then } |\varphi_k(x)|=O(|x|^{-s}) \mbox{ when } x\to\infty.
\]
Under assumption (A2) this result will imply that $\varphi_k\in
L^1(\mathbb{R})$, thus $Z_k$ has a bounded density function $f_k$ with
$\|f_k\|_\infty\leq C_k:=\int_\mathbb{R} |\varphi_k(x)|\,dx<\infty$. This gives us
\begin{eqnarray*}
\mathbb{E}( |A Z_k+B|^{-d}) &=& \int_\mathbb{R} \frac{f_k(x)}{|A x+
B|^d}\, dx \\
& = & |A|^{-d} \int_\mathbb{R} \frac{f_k(x)}{|x+ \fraca{B}{A}|^d}\, dx\\
&=& |A|^{-d} \int_\mathbb{R} \frac{f_k(u-\fraca{B}{A})}{|u|^d}\, du
\\
&=& |A|^{-d} \biggl( \int_{|u|> 1} \frac{f_k(u-\fraca{B}{A})}{|u|^d}\, du + \int_{|u|\leq1} \frac{f_k(u-\fraca{B}{A})}{|u|^d}\, du
\biggr)\\
& \leq& |A|^{-d} \cdot \biggl( 1+C_k \int_{|u|\leq1} \frac
{1}{|u|^d}\, du  \biggr).
\end{eqnarray*}

(ii) First we assume that $Z_k$ and $Z_l$ have a bounded joint density
function~$f$ with $\|f\|_\infty=C<\infty$, then
\begin{eqnarray*}
& &\mathbb{E} \bigl( (|A_1 Z_k+B_1|^2+|A_2 Z_l+B_2|^2)^{-d/2} \bigr) \\
&& \qquad  =  \iint\frac{f(x,y)}{(|A_1 x+B_1|^2+|A_2 y+B_2|^2)^{d/2}} \, dx\, dy\\
&& \qquad  =  |A_2|^{-d} \iint\frac{f(x,y)}{(|\fracd{A_1}{A_2} x+\fraca
{B_1}{A_2}|^2+|y+\fraca{B_2}{A_2}|^2)^{d/2}}\, dx\, dy \\
&& \qquad  \leq |A_2|^{-d} \frac{|A_2|}{|A_1|} \iint\frac{f(\fracd
{A_2}{A_1}u-\fraca{B_1}{A_1},v-\fraca{B_2}{A_2} )}{(u^2+v^2)^{d/2}}\,
du\, dv\\
&& \qquad \leq |A_1|^{-1}|A_2|^{-d+1} \biggl ( 1+C \iint
_{|u|^2+|v|^2\leq1} \frac{1}{(u^2+v^2)^{d/2}}\, du\,dv  \biggr),
\end{eqnarray*}
which gives us the conclusion. So it is enough to show that the
characteristic function
\[
\varphi\dvtx  (x,y)\in\mathbb{R}^2\mapsto\varphi(x,y)=\mathbb{E}\bigl(e^{i(xZ_k+yZ_l)}\bigr)
\]
is in $L^1(\mathbb{R}^2)$. For we consider the polar coordinates: for
$r \in\mathbb{R}_+$ and $\theta\in[0,2\pi)$ define
%
\begin{equation}\label{phirtheta}
\overline{\varphi}(r,\theta)=\varphi(r\cos\theta,r\sin\theta)=\mathbb
{E}\bigl(e^{i(r\cos\theta Z_k+r\sin\theta Z_l)}\bigr).
\end{equation}
Let $\psi(r)=\sup_{\theta\in[0,2\pi)} |\overline{\varphi}(r,\theta)|$.
Clearly $\psi(r) \leq1$, so it is enough to show that $\psi
(r)=O(r^{-s})$ for some $s>2$ when $r\to\infty$. This can be done by
using a~similar argument as in (i): from \eqref{phirtheta} and \eqref
{WZ} one has
\[
\overline{\varphi}(r,\theta)=\mathbb{E}\Biggl (\prod_{j=0}^{b-1} \overline
{\varphi} \bigl(r\cdot\overline{W}(j), \theta+\overline{\theta}(j)
\bigr) \Biggr),
\]
where $\overline{W}(j)=\sqrt{|W_k(j)|^2+|W_l(j)|^2}$ and $\overline
{\theta}_j=\arccos(W_k(j)/\overline{W}(j))$. This gives us
%
\begin{equation}\label{psi}
\psi(r) \leq\mathbb{E} \bigl(\psi(r\cdot\overline{W} ) \bigr)^b
 \qquad \mbox{where } \overline{W}=\sqrt{|W_k|^2+|W_l|^2}.
\end{equation}
Again, starting from inequality \eqref{psi} and following the proof of
Theorem 2.1 in~\cite{Liu01} (with a nontrivial modification which we
will present later), one can prove the following result:
%
\begin{equation}\label{r-s}\qquad
\mbox{If } \mathbb{E}(\overline{W}^{\,-s})<\infty\mbox{ for some } s>0,
\mbox{ then } \psi(r)=O(r^{-s}) \mbox{ when } r\to\infty.
\end{equation}
Then we can get the conclusion due to assumption (A2).

The nontrivial modification for proving \eqref{r-s} is the part that
proves \mbox{$\psi(r)<1$} holds for all $r>0$, the rest of the proof will
follow easily from the proof of Theorem 2.1 in~\cite{Liu01}. In order
to prove that $\psi(r)<1$ holds for all $r>0$, first we show that $\psi
(r)<1$ holds for all $r$ small enough.

Suppose that it is not the case. Then we can find sequences $r_n\to0$
and $\theta_n\in[0,2\pi)$ such that $|\overline{\varphi}(r_n,\theta
_n)|=1$, and thus there exists a subset $\Omega'\subset\Omega$ with
$\mathbb{P}(\Omega')=1$ and a sequence $\zeta_n\in[0,2\pi)$ such that
\[
r_n\cos\theta_n Z_k(\omega)+r_n\sin\theta_n Z_l(\omega)\in\zeta_n +
2\pi\mathbb{Z}
\]
holds for all $n\geq1$ and $\omega\in\Omega'$. In other words,
for any $\omega,\omega' \in\Omega'$ one has
\[
r_n\cos\theta_n \bigl(Z_k(\omega)-Z_k(\omega')\bigr)+r_n\sin\theta_n \bigl(Z_l(\omega
)-Z_l(\omega')\bigr)\in2\pi\mathbb{Z} \qquad   \forall n\geq1.\vadjust{\goodbreak}
\]
From $r_n\to0$ one gets
\[
\cos\theta_n \bigl(Z_k(\omega)-Z_k(\omega')\bigr)+\sin\theta_n\bigl (Z_l(\omega
)-Z_l(\omega')\bigr)=0
\]
for all $n$ large enough. Since $\cos\theta_n$ and $\sin\theta_n$
cannot be equal to $0$ at the same time and $Z_k$, $Z_l$ are not almost
surely a constant, there exist a subset $\Omega''\subset\Omega'$ with
$\mathbb{P}(\Omega'')=1$ and a constant $c\neq0$ such that
\[
Z_k(\omega)-Z_k(\omega')=c\bigl(Z_l(\omega)-Z_l(\omega')\bigr)
\]
holds for all $\omega,\omega' \in\Omega''$. This implies that $Z_k-c
Z_l$ is a constant on $\Omega''$. In other words, $\sum_{j=0}^{b-1}
W_k(j)Z_k(j)-c W_l(j)Z_l(j)$ is almost surely a constant. But this
could happen only if $W_k(j)Z_k(j)-c W_l(j)Z_l(j)$ is almost surely
equal to $0$ for each $j=0,\ldots ,b-1$ (since they are i.i.d. random
variables). So we get $c=1$ and $W_k=W_l$ almost surely, which is
contradictory to the assumption $\mathbb{P}(W_1=W_2)<1$.

Now suppose that there exists an $h>0$ such that $\psi(h)=1$, and we
may assume that $\psi(r)<1$ holds for all $0<r<h$. From \eqref{psi} we get
\[
1=\psi(h)\leq\mathbb{E} \bigl(\psi(h\cdot\overline{W} ) \bigr)^b
\leq1.
\]
This implies that almost surely $\psi (h\cdot\overline{W} )=1$.
Due to (A1) there exists $q\in(1,2]$ such that $\mathbb
{E}(|W_1|^{q})\vee\mathbb{E}(|W_2|^q)<b^{-1}$. Since $q/2<1$, by using
subadditivity of $x\rightarrow x^{q/2}$, we get that
\begin{eqnarray*}
\mathbb{P}(\overline{W} \geq1) &\leq& \mathbb
{E}\bigl((|W_k|^2+|W_l|^2)^{q/2}\bigr)\\
&\leq& \mathbb{E}(|W_k|^q+|W_l|^q)\\
&<&2b^{-1}\\
&\leq&1.
\end{eqnarray*}
Thus there exists $\delta<1$ such that $\psi(h\cdot\delta) =1$, which
is a contradiction.

\section*{Acknowledgments}

The author would like to gratefully thank the referee for his careful
reading of the original manuscript and for his many useful comments and
suggestions.

%

\printaddresses

\end{document}